\documentclass[11pt]{amsart}

\usepackage{amssymb}

\input xy

\xyoption{all}

\theoremstyle{plain}

\newtheorem{theorem}{Theorem}[section]

\newtheorem{proposition}[theorem]{Proposition}

\newtheorem{fact}[theorem]{Fact}

\newtheorem{claim}[theorem]{Claim}

\theoremstyle{definition} %actually mainly don't "mute" (in hebrew, in other words, be strait)

\newtheorem{definition}[theorem]{Definition}

\newtheorem{remark}[theorem]{Remark}

\theoremstyle{remark}

\newtheorem{hypothesis}[theorem]{Hypothesis}

\renewcommand{\phi}{\varphi}

\newcommand{\initial}\lessdot

\newcommand{\N}{\mathcal N}

\def\?{?\vadjust

{\vbox to 0pt{\vskip-7pt\hbox to 1.1\hsize{\hfill\huge ?!}}}}

\usepackage{verbatim}

\newcommand{\be}{\begin{enumerate}}
\newcommand{\ee}{\end{enumerate}}
\newcommand{\case}{\emph}

\usepackage{enumerate}
\renewcommand{\epsilon}{\varepsilon}
\newcommand{\nf}{\allowbreak}
\newtheorem{the main theorem}{The main theorem}[section]

 \def\nfork{\setbox0\hbox{$\bigcup$}%
 \setbox1=\hbox to \wd0{\hfil\vrule width 0.7pt depth 2pt height 7.5pt\hfil}%
 \wd1=0cm\relax\box1\box0}
 \def\dnf{\mathop{\nfork}\limits}

\begin{document}
%%%%%%%%%%%%%%%%%%%%%%%%%%%%%%%%
%%%%%%%%%%%%%%%%%%%%%%%%%%%%%%%%
\title{Non Forking Good Frames Without Local Character}

\author{Adi Jarden}
\address{Department of Mathematics.\\ Bar-Ilan University \\ Ramatgan 52900, Israel}
\email[Adi Jarden]{jardenadi@gmail.com}
%\address{Institute of Mathematics\\ Hebrew University\\ Jerusalem, Israel}

\author{Saharon Shelah}
\address{Institute of Mathematics\\ Hebrew University\\ Jerusalem, Israel\\ and Rutgers University\\
 New Brunswick\\ NJ, U.S.A}
\email[Saharon Shelah]{shelah@math.huji.ac.il}

%\date{October 6, 2008}

\maketitle

%\begin{abstract}
%We continue \cite{sh600}, studying stability theory for abstract
%elementary classes. But instead of assuming local character, like
%in \cite{sh600}, we assume \emph{weak} local character. We prove
%that if we "fulfill" the frame then it will satisfy local
%character. This theorem has an important application. In
%\cite{she46} it was proved that (under mild set theoretic
%assumptions) Categoricity in $\lambda,\lambda^+$ and medium number
%of models in $K_{\lambda^{++}}$ implies existence of an almost
%good $\lambda$-frame. So by our theorem, we can get the local
%character too. Therefore by assumptions on the number of models in
%$\frak{k}_\lambda,\frak{k}_{\lambda^+},\frak{k}_{\lambda^{++}}$
%and more mild assumptions, we can get existence of a good
%$\lambda$-frame. Combining this with \cite{sh600}, we almost
%conclude that the function $\lambda \to I(\lambda,K)$, which
%corresponds to each cardinal $\lambda$, the number of models in K
%of cardinality $\lambda$, is not arbitrary. This paper is not self
%contained, as it is based on
%\cite{sh600},\cite{jrsh875},\cite{sh838} and \cite{she46}.
%\end{abstract}

\begin{abstract}
We continue \cite{shh}.II, studying stability theory for abstract
elementary classes. In \cite{she46}, Shelah obtained a non-forking
relation for an AEC, $(K,\preceq)$, with $LST$-number at most
$\lambda$, which is categorical in $\lambda$ and $\lambda^+$ and
has less than $2^{\lambda^+}$ models of cardinality
$\lambda^{++}$, but at least one. This non-forking relation
satisfies the main properties of the non-forking relation on
stable first order theories, but only a weak version of the local
character.

Here, we improve this non-forking relation such that it satisfies
the local character, too. Therefore it satisfies the main
properties of the non-forking relation on superstable first order
theories.

Using results of \cite{shh}.II, we conclude that the function
$\lambda \to I(\lambda,K)$, which assigns to each cardinal
$\lambda$, the number of models in K of cardinality $\lambda$, is
not arbitrary.
\end{abstract}

\tableofcontents

\section{Preliminaries}
Familiarity with AEC's is assumed.
\begin{hypothesis}
\mbox{}
\begin{enumerate}
\item $(K,\preceq)$ is an AEC. \item $\lambda$ is a cardinal.
\item The Lowenheim Skolem Tarski number of $(K,\preceq)$,
$LST(K,\preceq)$, is at most $\lambda$.
\end{enumerate}
\end{hypothesis}

%\begin{definition}
%We say that $(K_\lambda,\preceq \restriction K_\lambda)$ has
%disjoint amalgamation when: If $M_0 \preceq M_1,M_0 \preceq M_2$
%then for some $M_3,f_1,f_2$ the following hold:
%\begin{enumerate}
%\item $M_0 \preceq M_3$. \item For $n=1,2$, $f_n$ is an
%$\preceq$-embedding from $M_n$ to $M_3$ over $M_0$. \item
%$f_1[M_1] \bigcap f_2[M_2]=M_0$.
%\end{enumerate}
%\end{definition}

%%%%%%%%%%%%%%%%
% We don't have to discuss universal models, so I delete (or not):
%%%%%%%%%%%%%%%%

\begin{definition}
Suppose $M_0\prec N$ in $K_\lambda$. We say that $N$ is
\emph{universal} over $M_0$ if for every $M_1\succ M_0$, there is
an embedding of $M_1$ into $N$ over $M_0$, namely, that fixes
$M_0$.
\end{definition}

The following proposition is a version of Fodor's Lemma (there is
no mathematical reason to choose this version, but we think that
it is convenient).
\begin{proposition} \label{1.11}
There exist no $\langle M_\alpha:\alpha \in \lambda^+ \rangle,\
\langle N_\alpha:\alpha \in \lambda^+ \rangle,\ \langle
f_\alpha:\alpha \in \lambda^+ \rangle,S$ such that the following
conditions are satisfied:
\begin{enumerate}
\item The sequences $\langle M_\alpha:\alpha \in \lambda^+
\rangle,\ \langle N_\alpha:\alpha \in \lambda^+ \rangle$ are
$\preceq$-increasing continuous sequences of models in
$K_\lambda$. \item For every $\alpha<\lambda^+$,
$f_\alpha:M_\alpha \to N_\alpha$ is a $\preceq$-embedding. \item
$\langle f_\alpha:\alpha \in \lambda^+ \rangle$ is an increasing
continuous sequence. \item $S$ is a stationary subset of
$\lambda^+$. \item For every $\alpha \in S$,
there is $a \in M_{\alpha+1}-M_\alpha$ %(or even in $M_{\lambda^+}-M_{\alpha}$)
such that $f_{\alpha+1}(a) \in N_\alpha$.
\end{enumerate}
\end{proposition}

\begin{proof}
Suppose there are such sequences. Denote $M=\bigcup
\{f_\alpha[M_\alpha]:\alpha \in \lambda^+\}$. By clauses (4),(5),
$||M||=K_{\lambda^+}$. $\langle f_\alpha[M_\alpha]:\alpha \in
\lambda^+ \rangle$, $\langle N_\alpha \bigcap M:\alpha \in
\lambda^+ \rangle$ are filtrations of $M$. So they are equal on a
club of $\lambda^+$. Hence there is $\alpha \in S$ such that
$f_\alpha[M_\alpha]=N_\alpha \bigcap M$. Hence $f_\alpha[M_\alpha]
\subseteq N_\alpha \bigcap f_{\alpha+1}[M_{\alpha+1}] \subseteq
N_\alpha \bigcap M = f_\alpha[M_\alpha]$ and so this is a chain of
equivalences. Especially $f_{\alpha+1}[\allowbreak M_{\alpha+1}]
\bigcap N_\alpha = f_\alpha[M_\alpha]$, in contradiction to
condition (5).
\end{proof}

\section{Non-forking frames}
The following definition, Definition \ref{definition good frame}
is an axiomatization of the non-forking relation in a superstable
first order theory. If we subtract axiom \ref{definition good
frame}(3)(c), we get the basic properties of the non-forking
relation in $(K_\lambda,\preceq \restriction K_\lambda)$ where
$(K,\preceq)$ is stable in $\lambda$.

Sometimes we do not find a natural independence relation with
respect to all the types. So first we extend the notion of an AEC
in $\lambda$ by adding a new function $S^{bs}$ which assigns a
collection of basic (because they are basic for our construction)
types to each model in $K_\lambda$, and then add an independence
relation $\dnf$ on basic types.

We do not assume \emph{the amalgamation property} in general, but
we assume the amalgamation property in $(K_\lambda, \preceq
\restriction K_\lambda)$. This is a reasonable assumption, because
it is proved in \cite{shh}.I, that if an AEC is categorical in
$\lambda$ and the amalgamation property fails in $\lambda$, then
under a plausible set theoretic assumption, there are
$2^{\lambda^+}$ models in $K_{\lambda^+}$.

\begin{definition} \label{2.1a} \label{definition good frame}
\label{definition of a good frame}
$\frak{s}=(K,\preceq,S^{bs},\dnf)$ is a \emph{good
$\lambda$-frame} if:

\begin{enumerate}
%(0)
\item $(K,\preceq)$ is an AEC in $\lambda$.

%(1)
\item
\begin{enumerate}[(a)]
\item $(K,\preceq)$ %It has a superlimit model,
satisfies the joint embedding property. \item $(K,\preceq)$
satisfies the amalgamation property. \item There is no
$\preceq$-maximal model in $K$.
\end{enumerate}

%(2)
\item $S^{bs}$ is a function with domain $K$, which satisfies the
following axioms:
\begin{enumerate}[(a)]
\item $S^{bs}(M) \subseteq S^{na}(M)=\{ga-tp(a,M,N):M\prec N \in
K,\ a\in N-M\}$. \item $S^{bs}$ respects isomorphisms: if
$ga-tp(a,M,N) \in S^{bs}(M)$ and $f:N \to N'$ is an isomorphism,
then $ga-tp(f(a),f[M],N') \in S^{bs}(f[M])$. \item Density of the
basic types: if $M,N \in K_\lambda$ and $M \prec N$, then there is
$a \in N-M$ such that $ga-tp(a,M,N) \in S^{bs}(M)$. \item Basic
stability: for every $M \in K$, the cardinality of $S^{bs}(M)$ is
$\leq \lambda$.
\end{enumerate}

%(3)
\item the relation $\dnf$ satisfies the following axioms:
\begin{enumerate}[(a)]
\item $\dnf$ is set of quadruples $(M_0,M_1,a,M_3)$ where
$M_0,M_1,M_3 \in K$, $a \in M_3-M_1$ and for $n=0,1$
$ga-tp(a,M_n,M_3) \in S^{bs}(M_n)$ and it respects isomorphisms:
if $\dnf(M_0,M_1,a,M_3)$ and $f:M_3 \to M_3'$ is an isomorphism,
then $\dnf(f[M_0],f[M_1],f(a),M_3')$. \item Monotonicity: if $M_0
\preceq M^*_0 \preceq M^*_1 \preceq M_1 \preceq M_3 \preceq
M_3^*,\ M^*_1\bigcup \{a\} \subseteq M^{**}_3 \preceq M^*_3$, then
$\dnf(M_0,M_1,a,M_3) \Rightarrow \dnf(M^*_0,M^*_1,a,M^{**}_3)$.
From now on, `$p \in S^{bs}(N)$ does not fork over $M$' will be
interpreted as `for some $a,N^+$ we have $p=ga-tp(a,N,N^+)$ and
$\dnf(M,N,a,N^+)$'. See Proposition \ref{dnf not as a sign}. \item
Local character: for every limit ordinal $\delta<\lambda^+$ if
$\langle M_\alpha:\alpha \leq \delta \rangle$ is an increasing
continuous sequence of models in $K_\lambda$, and
$ga-tp(a,M_\delta,M_{\delta+1}) \in S^{bs}(M_\delta)$, then there
is $\alpha<\delta$ such that $ga-tp(a,M_\delta,\allowbreak
M_{\delta+1})$ does not fork over $M_\alpha$. \item Uniqueness of
the non-forking extension: if $M,N \in K$, $M \preceq N$, $p,q \in
S^{bs}(N)$ do not fork over $M$, and $p\restriction
M=q\restriction M$, then $p=q$. \item Symmetry: if $M_0,M_1,M_3
\in K_\lambda$, $M_0 \preceq M_1 \preceq M_3,\ a_1 \in M_1,\
\linebreak ga-tp(a_1,M_0,\allowbreak M_3) \in S^{bs}(M_0)$, and
$ga-tp(a_2,M_1,M_3)$ does not fork over $M_0$, \emph{then} there
are $M_2,M^*_3 \in K_\lambda$ such that $a_2 \in M_2,\ M_0 \preceq
M_2 \preceq M^*_3,\ M_3 \preceq M^*_3$, and $ga-tp(a_1,M_2,M^*_3)$
does not fork over $M_0$. \item Existence of non-forking
extension: if $M,N \in K$, $p \in S^{bs}(M)$ and $M\prec N$, then
there is a type $q \in S^{bs}(N)$ such that $q$ does not fork over
$M$ and $q\restriction M=p$. \item Continuity: let
$\delta<\lambda^+$ and $\langle M_\alpha:\alpha \leq \delta
\rangle$ be an increasing continuous sequence of models in $K$ and
let $p \in S(M_\delta)$. If for every $\alpha \in \delta,\
p\restriction M_\alpha$ does not fork over $M_0$, then $p \in
S^{bs}(M_\delta)$ and does not fork over $M_0$.
\end{enumerate}
\end{enumerate}
\end{definition}

\begin{proposition}\label{dnf not as a sign}
If $\dnf(M_0,M_1,a,M_3)$ and the types
$ga-~tp(b,M_1,M_3^{*}),\allowbreak ga-~tp(a,M_1,M_3)$ are equal,
then we have $\dnf(M_0,M_1,a,M_3)$.
\end{proposition}

\begin{proof}
Since $ga-~tp(b,M_1,M_3^{*})=ga-~tp(a,M_1,M_3)$, there is an
amalgamation $(id_{M_3},f,M_3^{**})$ of $M_3$ and $M_3^*$ over
$M_1$ with $f(b)=a$. By Definition~ \ref{definition of a good
frame}(3)(b) (monotonicity) $\dnf(M_0,M_1,a,M_3^{**})$. Using
again Definition~ \ref{definition of a good frame}(3)(b), we get
$\dnf(M_0,M_1,a,f[M_3^*])$. Therefore by Definition
\ref{definition of a good frame}(3)(a), $\dnf(M_0,M_1,a,M_3^*)$.
\end{proof}

\begin{definition} \label{2.1b}\label{definition of an almost good
frame} \mbox{}

\begin{enumerate} \item $\frak{s}=(K,\preceq,S^{bs},nf)$ is an \emph{almost} good
$\lambda$-frame if $\frak{s}$ satisfies the axioms of a good
$\lambda$-frame except maybe local character, but $\frak{s}$
satisfies weak local character.

\item $\frak{s}$ satisfies \emph{weak local character} when there
is a 2-ary relation, $\prec^*$ on $K_\lambda$ which is included in
$\prec \restriction K_\lambda$ such that:
\begin{enumerate}
\item for each $M_0 \in K_\lambda$ there is $M_1 \in K_\lambda$
with $M_0 \prec^* M_1$, \item if $M_0 \prec^* M_1 \preceq M_2 \in
K_\lambda$ then $M_0 \prec^* M_2$, \item if $\langle
N_\alpha:\alpha<\delta+1 \rangle$ is a $\prec^*$-increasing
continuous sequence of models in $K_\lambda$, then for some $a \in
N_{\delta+1}$ and some ordinal $\alpha<\delta$,
$p=:ga-~tp(a,N_\delta,N_{\delta+1})$ is a basic type, which does
not fork over $N_\alpha$.
\end{enumerate}
\end{enumerate}
\end{definition}

In the following definition `na' means non-algebraic.
\begin{definition}
We define a function $S^{na}$ with domain $K_\lambda$ by
$S^{na}(M):=\{ga-~tp(a,M,N):M \preceq N,\ a \in N-M\}$.
\end{definition}

\begin{definition}
Let $\frak{s}$ be an almost good $\lambda$-frame. $\frak{s}$ is
\emph{full} if $S^{bs}=S^{na}$.
\end{definition}

%\begin{comment}
%%%%%%%%%%%%%%%5
%the proof of the stability theorem in \cite{jrsh875} applies also
%here. So I delete the following theorem: %%%%%%%%%%%%%%%%
%%%%%%%%%%%%%%
The following theorem says that the stability property in
$\lambda$ is satisfied and presents sufficient conditions for a
universal model. The stability in $\lambda$ can actually derived
from \cite[Theorem 2.20]{jrsh875}.
\begin{theorem} \label{4.0}
\mbox{}
\begin{enumerate}
\item Suppose:
\begin{enumerate}
\item $\frak{s}$ is an almost good $\lambda$-frame (so indirectly,
we assume basic stability).
   \item $\langle  M_\alpha:\alpha \leq \lambda \rangle$ is an increasing continuous
sequence of models in $K_\lambda$.
   \item $M_{\alpha+1}$ realizes $S^{bs}(M_\alpha)$.
   \item $M_\alpha\prec^*M_{\alpha+1}$.
\end{enumerate}
Then $M_{\lambda}$ is universal over $M_0$.\item There is a model
in $K_{\lambda}$ which is universal over $\lambda$. \item For
every $M \in K_\lambda$, $|S(M)| \leq \lambda$.
\end{enumerate}
\end{theorem}

%%%%%%%%%%%%5
%We are still in a comment or not! %%%%%%%%%%%%%%%%
%%%%%%%%%%%5%

\begin{proof}
Obviously $(1) \Rightarrow (2) \Rightarrow (3)$. Why does (1)
hold? We have to prove that letting $M_0\prec N$, $N$ can be
embedded in $M_\lambda$ over $M_0$. Toward a contradiction assume
that:
\begin{center}
(*) There is no an embedding from N into $M_\lambda$ over $M_0$.\\
\end{center}
Let $cd$ be a bijection from $\lambda \times \lambda$ onto
$\lambda$. Now we choose $N_\alpha,A_\alpha,\langle
a_{\alpha,\beta}:\beta<\lambda \rangle, f_\alpha$ by induction on
$\alpha$ such that:
\begin{enumerate}
\item $N_0=N,\ f_0=id_{M_0}$ \item $\langle
N_\alpha:\alpha<\lambda \rangle$ is an increasing continuous
sequence of models in $K_\lambda$. \item $\langle
f_\alpha:\alpha<\lambda \rangle$ is an increasing continuous
sequence of functions. \item $f_\alpha:M_\alpha \hookrightarrow
N_\alpha$ is an embedding. \item
$N_\alpha=\{a_{\alpha,\beta}:\beta<\lambda\}$. \item
$A_\alpha=\{cd(\gamma,\beta):\gamma \preceq  \alpha$,
$ga-~tp(a_{\gamma,\beta},f_\alpha[M_\alpha],N_\alpha) \in
S^{bs}(f_\alpha[M_\alpha])$\}. \item $a_{\gamma,\beta} \in
f_{\alpha+1}[M_{\alpha+1}]$ where
$(\gamma,\beta)=cd^{-1}(Min(A_\alpha))$.
\end{enumerate}

%%%%%%%%%%%%5
%We are still in a comment or not! %%%%%%%%%%%%%%%%
%%%%%%%%%%%5%

\case{Why can we carry out the induction?} For $\alpha=0$ or
limit, there is no problem. Suppose we have chosen
$N_\alpha,A_\alpha,\langle a_{\alpha,\beta}:\beta<\lambda
\rangle,f_\alpha$. If $f_\alpha[M_\alpha]=N_\alpha$, then
$f_\alpha^{-1} \restriction N_0$ is an embedding over $M_0$, in
contradiction to (*). Thus $f_\alpha[M_\alpha] \neq N_\alpha$.
Therefore there is a type in $S^{bs}(f_\alpha[M_\alpha])$ which
$N_\alpha$ realizes. Hence $A_\alpha \neq \emptyset$. So by the
definition of a
type, there is no problem to find $N_{\alpha+1},A_{\alpha+1},\langle a_{\alpha+1,\beta}:\beta<\lambda \rangle,f_{\alpha+1}$.\\
\case{Why is this enough?} Define $N_\lambda:=\bigcup
\{N_\alpha:\alpha<\lambda \},\ f_\lambda:=\bigcup
\{f_\alpha:\alpha<\lambda\}$. By smoothness, $f_\lambda[M_\lambda]
\preceq  N_\lambda$. But $f_\lambda[M_\lambda] \neq N_\lambda$
(otherwise $f_\lambda^{-1} \restriction N_0$ is an embedding over
$M_0$, in contradiction to (*)). So by \emph{weak local
character}, there is $c \in N_\lambda-f_\lambda[M_\lambda]$ and
there is a $\gamma \in \lambda$ such that
$ga-~tp(c,f_\lambda[M_\lambda],N_\lambda)$ does not fork over
$f_{\gamma}[M_{\gamma}]$. Without loss of generality, $c \in
N_{\gamma}$, because we can increase $\gamma$. Therefore there is
$\beta \in \lambda$ such that $c=a_{\gamma,\beta}$. Hence
$ga-~tp(a_{\gamma,\beta},f_\gamma[M_\gamma],N_\gamma) \in
S^{bs}(f_\gamma[M_\gamma])$. Define an injection
$g:[\gamma,\lambda) \rightarrow \lambda$ by
$g(\alpha):=$Min$(A_\alpha)$. For each $\alpha \in
[\gamma,\lambda)$, $cd(\gamma,\beta) \in A_\alpha$. So
$g(\alpha)<cd(\gamma,\beta)$, (otherwise by (7) $a_{\gamma,\beta}
\in f_{\alpha+1}[M_{\alpha+1}] \subset f_\lambda[M_\lambda]$, but
$a_{\gamma,\beta}=c \notin f_\lambda[M_\lambda]$), and $g$ is an
injection from $[\gamma,\lambda)$ to $cd(\gamma,\beta)$ which is
impossible. Thus (*) implies a contradiction.
\end{proof}
%%%%%%%%%%%%
%%%%%%%%%%%5
%\end{comment}

\section{Non-forking amalgamation}
\begin{hypothesis}
$\frak{s}$ is an almost good $\lambda$-frame.
\end{hypothesis}

In this section we present a theorem from \cite{jrsh875}, which
says that we can derive a non-forking relation on models, from the
non-forking relation on elements. First we have to define the
conjugation property.

\begin{definition} \label{2.2}
\mbox{ } \\
(1) Let $p=ga-~tp(a,M,N)$. Let $f$ be an isomorphism of $M$ (i.e.
$f$ is an injection with domain M, and the relations and functions
on $f[M]$ are defined such that $f:M \hookrightarrow f[M]$ is an
isomorphism). Define $f(p)=ga-~tp(f(a),f[M],f^+[N])$, where $f^+$
is an extension of $f$ (and the relations and functions on
$f^+[N]$ are defined such that $f^+:N \hookrightarrow f^+[N]$ is
an isomorphism).

(2) Let $p_0,p_1$ be types, $n<2 \rightarrow p_n \in S(M_n)$. We
say that $p_0,p_1$ are \emph{conjugate} if there is an isomorphism
$f:M_0 \hookrightarrow M_1$ such that $f(p_0)=p_1$.
\end{definition}

\begin{claim} \label{2.3}  \mbox{}
\begin{enumerate}
\item In Definition \ref{2.2}, $f(p)$ does not depend on the
choice of $f^+$. \item The conjugation relation is an equivalence
relation.
\end{enumerate}
\end{claim}

\begin{proof}
Easy.
\end{proof}

\begin{definition}\label{definition of conjugation}
Let $\frak{s}$ be an almost good $\lambda$-frame. $\frak{s}$ is
said to satisfy the \emph{conjugation property}, when: if $p \in
S^{bs}(M_1)$ does not fork over $M_0$, then there is an
isomorphism $f:M_1 \to M_0$ such that $f(p)=p \restriction M_0$.
\end{definition}

\begin{remark}
If $\frak{s}$ satisfies the conjugation property, then $K_\lambda$
is categorical.
\end{remark}

Now we present the properties that a non-forking relation should
satisfy.
\begin{definition} \label{5.1}
Let $NF \subseteq \ ^4K_\lambda$ be a relation. We say
$\bigotimes_{NF}$ when the following axioms are satisfied:
\begin{enumerate}[(a)]
\item If $NF(M_0,M_1,M_2,M_3)$, then $n\in \{1,2\}\rightarrow
M_0\preceq  M_n\preceq  M_3$ and $M_1 \cap M_2=M_0$.  \item The
monotonicity axiom: if $NF(M_0,M_1,M_2,M_3) \ and \ N_0=M_0, n<3
\rightarrow N_n\preceq  M_n\wedge N_0\preceq  N_n\preceq  N_3,
(\exists N^{*})[M_3\preceq  N^{*}\wedge N_3\preceq  N^{*}]$, then
$NF(N_0 \allowbreak ,N_1,N_2,N_3)$. \item The existence axiom: for
every $N_0,N_1,N_2 \in K_\lambda$, if $l\in \{1,2\}\rightarrow N_0
\preceq  N_l \ and \ N_1\bigcap N_2=N_0$, then there is $N_3$ such
that $NF(N_0,N_1,N_2,N_3)$. \item The uniqueness axiom: suppose
for $x=a,b$ we have $NF(N_0,N_1,N_2,N^{x}_3)$. Then there is a
joint embedding of $N^a,N^b \ over \ N_1 \bigcup N_2$. \item The
symmetry axiom: $NF(N_0,N_1,N_2,N_3) \leftrightarrow
NF(N_0,N_2,N_1,N_3)$. \item The long transitivity axiom: for
$x=a,b$, let $\langle M_{x,i}:i\preceq  \alpha^* \rangle$ be an
increasing continuous sequence of models in $K_\lambda$. Suppose
$i<\alpha^* \rightarrow NF(M_{a,i},M_{a,i+1},M_{b,i},M_{b,i+1})$.
Then $NF(M_{a,0},M_{a,\alpha^{*}},M_{b,0},M_{b,\alpha^{*}})$.
\end{enumerate}
\end{definition}

\begin{definition}
Let $NF$ be a relation such that $\bigotimes_{NF}$. We say that
$NF$ respects the frame $\frak{s}$ when: if $NF(M_0,M_1,M_2,M_3)$
and $ga-~tp(a,M_0,M_1) \in S^{bs}(M_0)$, then $ga-~tp(a,M_2,M_3)$
does not fork over $M_0$.
\end{definition}

\begin{theorem}\label{there is a $NF$ relation}
Suppose:
\begin{enumerate}
\item K is categorical in $\lambda$. \item $\frak{s}$ is an almost
good $\lambda$-frame which satisfies the conjugation property.
\item $I(\lambda^{++},K)<\mu_{unif}(\lambda^{++},2^{\lambda^+})$.
\item $2^\lambda<2^{\lambda^+}<2^{\lambda^{++}}$. \item The ideal
$WDmId(\lambda^+)$ is not saturated in $\lambda^{++}$.
\end{enumerate}

Then there is a relation $NF$ such that $\bigotimes_{NF}$ and $NF$
respects the frame~ $\frak{s}$.
\end{theorem}

\begin{proof}
By \cite{jrsh875}: by Corollary \cite[4.18]{jrsh875}, $K^{3,uq}$
is dense with respect to $\preceq_{bs}$. Hence by Theorem
\cite[5.15]{jrsh875}, there is a unique relation, $NF$, with
$\bigotimes_{NF}$. Now see Definition \cite[5.3]{jrsh875}.
\end{proof}

\section{A full good $\lambda$-frame}
\begin{hypothesis}
$\frak{s}$ is an almost good $\lambda$-frame which satisfies the
conjugation property.
\end{hypothesis}

\begin{definition}
$nf^{NF}:=\{(M_0,M_1,a,M_3):M_0,M_1,M_3 \in K_\lambda,\ M_0
\preceq M_1 \preceq M_3,\ a \in M_3-M_1$ and for some $M_2 \in
K_\lambda$, $M_0 \preceq M_2,\ a \in M_2-M_0$ and
$NF(M_0,M_1,M_2,M_3)$\}.
\end{definition}

The following theorem is similar to Claim \cite[9.5.2]{shh}.III.
\begin{theorem}
\label{5.16} Let $\frak{s}$ be an almost good $\lambda$-frame
which satisfies the conjugation property. Then
$\frak{s}^{NF}=(K,\preceq,S^{na},nf^{NF})$ is a full good
$\lambda$-frame.
\end{theorem}

\begin{proof}
We will prove the conditions in Definition \ref{definition good
frame}:

1. Trivial.

2. (a),(b),(c) are trivial. (d) (basic stability) is satisfied by
Theorem \ref{4.0}(3).

3. (a) is trivial.

(b) is OK by the monotonicity of $NF$, i.e. Definition
\ref{5.1}(b).

Axiom (c) (local character) is the heart of the matter. Let $j$ be
a limit ordinal, let $\langle N_i:i \preceq j+1\rangle $ be an
increasing continuous sequence of models in $K_\lambda$ and let
$p=:ga-~tp(c,N_j,N_{j+1}) \in S^{na}(N_j)$. We have to find $i<j$
such that $p$ does not fork over $N_i$ in the sense of $nf^{NF}$,
i.e. $nf^{NF}(N_i,c,N_j,N_{j+1})$. It is enough to find an
increasing continuous sequence $\langle M_i:i \preceq j\rangle $
such that for each $i \preceq  j,\ N_i \preceq M_i$ and
$NF(N_i,\allowbreak N_{i+1},\allowbreak M_i,\allowbreak M_{i+1})$
(so $NF(N_i,N_j,M_i,M_j))$ and $N_{j+1} \preceq M_j$ (for some
$i<j$ $c \in M_i$, so $nf^{NF}(N_i,c,N_j,N_{j+1})$). Without loss
of generality, $cf(j)=j$. We try to construct $\langle
N_{\alpha,i}:i \preceq j+1\rangle$ by induction on $\alpha \in
\lambda^+$, such that:
\begin{enumerate}
\item For each $\alpha \in \lambda^+,\ \langle N_{\alpha,i}:i
\preceq j+1\rangle $ is an increasing continuous sequence of
models in $K_\lambda$. \item For each $i \preceq  j,\ \langle
N_{\alpha,i}:\alpha<\lambda^+\rangle $ is an $\prec^*$-increasing
continuous sequence of models in $K_\lambda$ and $N_{\alpha,j+1}
\preceq  N_{\alpha+1,j+1}$. \item $N_{0,i}=N_i$. \item For each
$i<j$ and $\alpha<\lambda^+$, we have $NF(N_{\alpha,i},\allowbreak
N_{\alpha,i+1},\allowbreak N_{\alpha+1,i},\allowbreak
N_{\alpha+1,i+1})$.
\item For each $\alpha \in S=:\{\delta \in \lambda^+:cf(\delta)=j\}$, we have $N_{\alpha,j+1} \bigcap N_{\alpha+1,j} \neq N_{\alpha,j}$.\\
\end{enumerate}

\begin{comment}
********************
\begin{displaymath}
\begin{matrix}
N_{\alpha+1,i} & N_{\alpha+1,i+1} & . & .& . & N_{\alpha+1,j} &
N_{\alpha+1,j+1}\\
N_{\alpha,i} & N_{\alpha,i+1} & . & .& . & N_{\alpha,j} &
N_{\alpha,j+1}\\ \\N_i & N_{i+1} & . & .& . & N_j & N_{j+1}
\end{matrix}
\end{displaymath}
********************
\end{comment}

If we succeed, then by clauses (2) and (5), the quadruple
$$\langle N_{\alpha,j}:\alpha < \lambda^+\rangle ,\ \langle
N_{\alpha,j+1}:\alpha < \lambda^+ \rangle,\ \langle
id_{N_{\alpha_j}}:\alpha<\lambda^+ \rangle,\ S$$ forms a
counterexample to Claim \ref{1.11}, so it is impossible to carry
out this construction.

\case{Where will we get stuck?} For $\alpha=0$, we will not get
stuck, see item (3).

For $\alpha$ limit, just (1),(2) are relevant, and we just have to
take unions and use smoothness.

So we will get stuck at some successor ordinal. Suppose we have
defined $\langle N_{\alpha,i}:i \preceq j+1\rangle$. Can we find
$\langle N_{\alpha+1,i}:i \preceq j+1 \rangle$? If $\alpha \notin
S$, then it is easier, so assume $\alpha \in S$. Let $\langle
\beta(i):i \preceq  j+1\rangle $ be an increasing continuous
sequence of ordinals such that $\beta(j)=\alpha$. If
$N_{\alpha,j}=N_{\alpha,j+1}$, then we can define
$M_i:=N_{\alpha,i}$ and the local character is proved ($N_j
\preceq N_{\alpha,j}=M_j$, so see the beginning of the proof). So
without loss of generality, $N_{\alpha,j+1} \neq N_{\alpha,j}$.
\newpage

In the following diagram, the arrows describe the
$\prec^*$-increasing continuous sequence $\langle N_{\beta(i),i}:i
\preceq j\rangle ^\frown \langle N_{\alpha,j+1}\rangle $. A model
that appears at the right and above another model is bigger than
it.
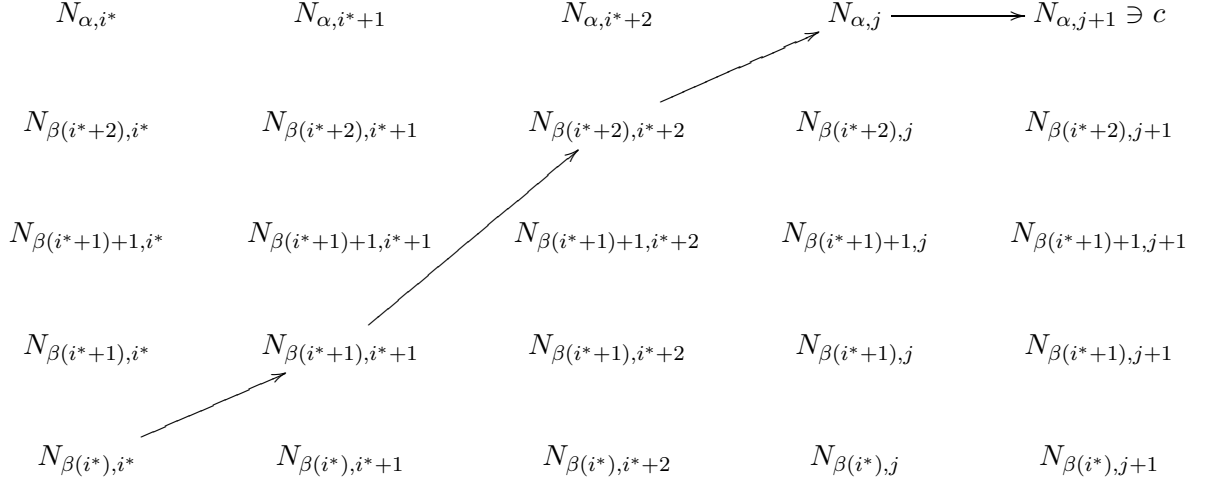
\begin{displaymath}
\xymatrix{ N_{\alpha,i^*} & N_{\alpha,i^*+1} & N_{\alpha,i^*+2} &
N_{\alpha,j} \ar[r] & N_{\alpha,j+1} \ni c\\
 N_{\beta(i^*+2),i^*} &
N_{\beta(i^*+2),i^*+1} & N_{\beta(i^*+2),i^*+2} \ar[ur] &
N_{\beta(i^*+2),j} & N_{\beta(i^*+2),j+1}\\ N_{\beta(i^*+1)+1,i^*}
& N_{\beta(i^*+1)+1,i^*+1} & N_{\beta(i^*+1)+1,i^*+2}&
N_{\beta(i^*+1)+1,j} & N_{\beta(i^*+1)+1,j+1}\\
N_{\beta(i^*+1),i^*} & N_{\beta(i^*+1),i^*+1} \ar[uur] &
N_{\beta(i^*+1),i^*+2}&
N_{\beta(i^*+1),j} & N_{\beta(i^*+1),j+1}\\
N_{\beta(i^*),i^*} \ar[ur] & N_{\beta(i^*),i^*+1} &
N_{\beta(i^*),i^*+2} &N_{\beta(i^*),j} & N_{\beta(i^*),j+1} }
\end{displaymath}

By weak local character, there is an element $c$ and an ordinal
$i^*$ such that $ga-~tp(c,N_{\alpha,j},N_{\alpha,j+1})$ does not
fork over $N_{\beta(i^*),i^*}$.

By Definition \ref{definition of a good frame}(b) (the
monotonicity axiom), $ga-~tp(c,N_{\alpha,j},N_{\alpha,j+1})$ does
not fork over $N_{\alpha,i^*+1}$ and so
$ga-~tp(c,N_{\alpha,i^*+1},N_{\alpha,j+1}) \in
S(N_{\alpha,i^*+1})$. So there is an increasing continuous
sequence $\langle N^{temp}_{\alpha+1,i}:i \preceq  j\rangle $ such
that for $i<j$ we have $NF(N_{\alpha,i},\allowbreak
N_{\alpha,i+1},\allowbreak N^{temp}_{\alpha+1,i},\allowbreak
N^{temp}_{\alpha+1,i+1})$, and there is $a \in
N^{temp}_{\alpha+1,i^*+1}$ such that
$ga-~tp(a,N_{\alpha,i^*+1},N^{temp}_{\alpha+1,i^*+1})=ga-~tp(c,N_{\alpha,i^*+1},N_{\alpha,j+1})$.
[Why? For $i\preceq  i^*$ define
$N^{temp}_{\alpha+1,i}=N_{\alpha,i}$. Choose
$N^{temp}_{\alpha+1,i^*+1}$ which is isomorphic to
$N_{\alpha,j+1}$ over $N_{\alpha,i^*+1}$ and
$N^{temp}_{\alpha+1,i^*+1} \bigcap
N_{\alpha,j+1}=N_{\alpha,i^*+1}$. For $i \in (i^*+1,j]$ choose
$N^{temp}_{\alpha+1,i+1}$ such that $NF(N_{\alpha,i},\allowbreak
N_{\alpha,i+1},\allowbreak N^{temp}_{\alpha+1,i},\allowbreak
N^{temp}_{\alpha+1,i+1})$. If $i$ is limit, then define
$N^{temp}_{\alpha+1,i}:=\bigcup
\{N_{\alpha+1,\epsilon}:\epsilon<i\}$]. Now by the long
transitivity of $NF$ we have $NF(N_{\alpha,i^*+1},\allowbreak
N_{\alpha,j},\allowbreak N^{temp}_{\alpha+1,i^*+1},\allowbreak
N^{temp}_{\alpha+1,j})$ and so since $NF$ respects $s$, the type
$ga-~tp(a,N_{\alpha,j},N^{temp}_{\alpha+1,j})$ does not fork over
$N_{\alpha,i^*+1}$. So by Definition 2.1(e), (the uniqueness of
the non-forking extension),
$ga-~tp(a,N_{\alpha,j},N^{temp}_{\alpha+1,j})=ga-~tp(c,N_{\alpha,j},N_{\alpha,j+1})$.
Hence by the definition of the equality between types, without
loss of generality, there is a model $N_{\alpha+1,j+1}$ such that
$N_{\alpha,j+1} \preceq N_{\alpha+1,j+1}$, there is an embedding
$f:N^{temp}_{\alpha+1,j} \hookrightarrow N_{\alpha+1,j+1}$ over
$N_{\alpha,j}$ and $f(a)=c$. Now for $i \preceq  j$ define
$N_{\alpha+1}:=f[N^{temp}_{\alpha+1,i}]$. Why is (5) satisfied? $c
\in N_{\alpha,j+1} \bigcap N_{\alpha+1,i+1}-N_{\alpha,i+1}$. By
(4) and the long transitivity of $NF$, we have
$NF(N_{\alpha,i+1},N_{\alpha,j},N_{\alpha+1,i+1},N_{\alpha+1,j})$,
so $c \notin N_{\alpha,j}$, but since $N_{\alpha+1,i+1} \subset
N_{\alpha+1,j}$ we have $c \in N_{\alpha+1,j}$. Hence $c \in
N_{\alpha,j+1} \bigcap N_{\alpha+1,j}-N_{\alpha,j}$) Hence we can
carry out the construction.

%d. Transitivity: Suppose $p \in S^{na}(M_2)$ does not fork over
%$M_1$ and $p \restriction M_1$ does not fork over $M_0$. Then
%there are $M_0^*,M_{1,1}^*,M_{1,2}^*,M_2^*$ such that $NF(M_0,\nf
%M_1,\nf M_0^*,\nf M_{1,1}^*),NF(M_1,\allowbreak M_2,\nf
%M_{1,2}^*,\nf M_2^*)$,\nf $a_1 \in M_0^*$, $a_2 \in M_{1,2}^*$,
%$ga-~tp(a_1,M_1,M_{1,1}^*)=p \restriction M_1$ and
%$ga-~tp(a_2,M_2,M_2^*)=p$.

%\begin{displaymath}
%\xymatrix{ &&M_1^* \ar[r] & M\\
%M_0^* \ar[r] & M_{1,1}^* \ar[ur]^{f_1} & M_{1,2}^* \ar[u]^{f_2}
%\ar[r] & M_2^* \ar[u]^g\\
%M_0 \ar[u] \ar[r] & M_1 \ar[u] \ar[ur] \ar[r] & M_2 \ar[ur] }
%\end{displaymath}

%As for $n=1,2$, $ga-~tp(a_n,M_1,M_{1,n}^*)$ does not depend on $n$,
%there are $M_1^*,f_1,f_2$ such that $f_n:M_{1,n}^* \hookrightarrow
%M_1^*$ is an embedding over $M_1$ and $f_1(a_1)=f_2(a_2)$. So
%there are $M,g$ such that $g:M_2^* \hookrightarrow M,\ f_2 \subset
%g$ is an embedding and $NF(f_2[M_{1,2}^*],\nf g[M_2^*],\nf M_1^*,
%\nf M)$. So by the long transitivity $NF(M_1,\nf g[M_2],\nf
%M_1^*,\nf M)$. By the monotonicity axiom $NF(M_0,\nf M_1,\nf
%f_1[M_0^*],\nf M_1^*)$. So by the long transitivity $NF(M_0,\nf
%g[M_2],\nf f_1[M_0^*],\nf M)$. So
%$ga-~tp(g(a_2),\nf g[M_2],M)$ does not fork over $M_0$. So $p$ does not fork over $M_0$.\\
(d) Uniqueness: suppose for $n<2,\ ga-~tp(a^n,M_0,M_1^n)$ does not
depend on $n$, and $NF(M_0,\nf M_2,\nf M_1^n,\nf M_3^n)$, see the
diagram below. We have to prove that $ga-~tp(a^n,M_2,M_3^n)$ does
not depend on $n$. By the definition of the equality between
types, there is an amalgamation $f^0,f^1,M_1$ of $M_1^0,M_1^1$
over $M_0$. So there are models $M_3^{n,+}$ and embeddings
$f_n^+:M_3^n \hookrightarrow M_3^{n,+}$, such that for $n<2$ we
have $NF(f_n[M_1^n],\nf f_n^+[M_3^n],\nf M_1,\nf M_3^{n,+})$ and
$f_n \subset f_n^+$. Since $M_2 \bigcap M_1^n=M_0$, without loss
of generality, $f_n^+ \restriction M_2=id_{M_2}$ (we can change
the names of the elements in $M_2-M_0$, i.e. $M_2-M_1^n$). By the
long transitivity axiom of $NF$, we have $NF(M_0, \nf M_2,\nf
M_1,\nf M_3^{n,+})$. So by the uniqueness of $NF$, there is a
joint embedding $g^0,g^1,M_3$ of $M_3^{0,+},M_3^{1,+}$ over $M_1
\bigcup M_2$. So $g^0 \circ f_0^+,\nf g^1 \circ f_1^+,\nf M_3$ is
an amalgamation of $M_3^0,\nf M_3^1$ over $M_2$. Since $a_n \in
M_1^n$, $(g^n \circ f_n^+)(a_n)=f_n(a_n)$ and so it does not
depend on $n$ (since $f_0,f_1$ are witnesses for
$ga-~tp(a_1,M_0,M_1^n)$ does not depend on $n$). So
$ga-~tp(a^n,M_2,M_3^n)$ does not depend on $n$.

\begin{displaymath}
\xymatrix{ M_1 \ar[r] & M_3^{n,+} \ar[r]^{g^n} & M_3\\
M_1^n \ar[u]^{f_n} \ar[r] & M_3^n \ar[u]^{f_n^+}\\
M_0 \ar[u] \ar[r] & M_2 \ar[u] }
\end{displaymath}

(e) Symmetry: by the symmetry of $NF$, i.e. Definition
\ref{5.1}(e).

(f) By the corresponding axiom of $NF$, i.e. Definition
\ref{5.1}(c).

(g) Continuity: it is easy to see that continuity follows by local
character, because by definition, $s^{NF}$ is full.
\end{proof}

Now we can present the main theorem: we get a good
$\lambda$-frame.
\begin{theorem}
\label{11.0} Let $(K,\preceq)$ be an AEC such that:
\begin{enumerate}
\item $K$ is categorical in $\lambda$,$\lambda^+$ and $1 \leq
I(\lambda^{+2},K)<\mu_{unif}(\lambda^{+2},2^{\lambda^+})$. \item
$2^\lambda<2^{\lambda^+}<2^{\lambda^{+2}}$, and $WDmId(\lambda^+)$
is not saturated in $\lambda^{+2}$.
\end{enumerate}
\underline{Then:}\\
there is an almost good $\lambda$-frame, $\frak{s}$ with
complete... $(K_\frak{s},\preceq_\frak{s})=((K_\lambda,\preceq)$
and a type is basic if it is minimal. Moreover, if $\frak{s}$
satisfies the conjugation property, then there is a good
$\lambda$-frame with $(K_s,\preceq_s)=((K,\preceq )$.
\end{theorem}

\begin{remark}
Background on Weak Diamond appears in \cite{ds} and in Chapter 13
of \cite{grbook}. Concerning $\mu_{unif}(\mu^+,2^\mu)$, see the
last chapter of \cite{shh}, \cite{jrsh875} or \cite{jrsh966}. It
is "almost $2^{\mu^+}$": $1<\mu_{unif}(\mu^+,2^\mu)$, If
$\beth_\omega \preceq  \mu$, then
$\mu_{unif}(\mu^+,2^\mu)=2^{\mu^+}$ and in any case it is not
clear if $\mu_{unif}(\mu^+,2^\mu)<2^{\mu^+}$ is consistent. There
are more claims which say that it is a "big cardinal".
\end{remark}

\begin{proof}
By Theorem \cite[0.2]{she46} there is such an almost good frame.
So by Theorem \ref{5.16} we have the ``moreover''.
\end{proof}

While in \cite{shh}.II we obtained a good $\lambda^+$-frame, here
we obtained a $\lambda$-good frame. Why is this important? In
Section 1 of \cite{shh}.III, Shelah defined weakly dimensionality
of a good frame, and proved that it is equal to the categoricity
in the successor cardinal. Since here we assume categoricity in
$\lambda^+$, the good $\lambda$-frame we obtained here is weakly
dimensional.

\section{The function $\lambda \to I(\lambda,K)$ is not arbitrary}
In this section, we prove, under set theoretical assumptions, that
there is no AEC, $(K,\preceq)$, which is categorical in
$\lambda,\lambda^+...\lambda^{+(n-1)}$, but has no model of
cardinality $\lambda^{+n}$. The main results of Section 4 enables
to prove only a weaker version of this theorem. But we can prove
this theorem, using results of \cite{she46} and \cite{shh}.II.

By the last section in \cite{shh}.II (alternatively, see Corollary
\cite[12.6]{jrsh875}):
\begin{fact}\label{11.2}
Suppose:
\begin{enumerate}
\item $n<\omega$, \item $\frak{s}=(K,\preceq,S^{bs},nf)$ is a good
$\lambda$-frame , \item For each $m<n$,
$I(\lambda^{+(2+m)},K)<\mu_{unif}(\lambda^{+(2+m)},2^{\lambda^{+(1+m)}})$,
\item
$2^\lambda<2^{\lambda^+}<2^{\lambda^{+2}}<...<2^{\lambda^{+(1+n)}}$,
 \item For each $m<n$, the ideal $WDmId(\lambda^{+1+m})$ is not saturated in
$\lambda^{+(2+m)}$.
\end{enumerate}
\emph{Then}
\begin{comment}
%%%%%%%%%%%
%this is not needed:
there is a good $\lambda^{+n}$-frame
$s^n=:((k^n,\preceq ^n),S^{bs,+n},nf^{+n})$, such that:
%\begin{enumerate}
\item $K^n_{\lambda^{+n}} \subset K_{\lambda^{+n}},\ \preceq ^n
\subset \preceq ^k\restriction K^n$. \item  There is a model in
$K^n$ of cardinality $\lambda^{+(2+n)}$.
\end{enumerate}
%%%%%%%%%%%%%%
%%%%%%%%%%%%%%%%5 (it I want to cancel the decition that it is a comment, then I have to delete also the % before ``\begin{enumerate}'.
%%%%%%%%% second, I have to delete the following line:
\end{comment}
there is a model in $K$ of cardinality $\lambda^{+(2+n)}$.
\end{fact}

By the following theorem, if $f:card \to card$ is a class function
(from the cardinals to the cardinals) with
$f(\lambda)=f(\lambda^+)=...f(\lambda^{+(n-1)})=1$ and
$f(\lambda^{+n})=0$, then under specific set theoretical
assumptions (clauses (4),(5), below), $f$ cannot be the spectrum
of categoricity of any AEC.
\begin{theorem} \label{11.3}
There are no $K,\preceq,n,\lambda$ such that
\begin{enumerate}
\item $n\geq 3$ is a natural number,  \item $(K,\preceq)$ is an
AEC, \item $K$ is categorical in $\lambda^{+m}$ for each $m<n$,
but $K_{\lambda^{+n}}=\emptyset$, \item
$2^\lambda<2^{\lambda^+}<2^{\lambda^{+2}}<...<2^{\lambda^{+(n-1)}}$,
\item For each $m<n-2$, $WDmId(\lambda^{+1+m})$ is not saturated
in
$\lambda^{+(2+m)}$. %\item For each $m<n$,
%$I(\lambda^{+(2+m)},K)<\mu_{unif}(\lambda^{+(2+m)},2^{\lambda^{+(1+m)}})$.
\end{enumerate}
\end{theorem}

Before we prove Theorem \ref{11.3}, we prove a weaker version of
it:
\begin{proposition}
The same as Theorem \ref{11.3}, but here we assume, in addition,
that if $M_0 \preceq M_1 \preceq M_2$, $a \in M_2-M_1$ and
$ga-~tp(a,M_0,M_2)$ is minimal, then the types
$ga-~tp(a,M_1,M_2),ga-~tp(a,M_0,M_2)$ are conjugate.
\end{proposition}

\begin{proof}
By Theorem \ref{11.0}, there is a good $\lambda$-frame with
$(K_s,\preceq_s)=((K,\preceq )$. Hence by Fact $\ref{11.2}$, there
is a model in $K$ of cardinality $\lambda^{+(n)}$.
\end{proof}

\begin{remark}
To our opinion, by Claim \cite[7.4]{she46}(p.~76), it is
reasonable to assume that $\frak{s}$ satisfies the conjugation
property.
\end{remark}

Now we prove Theorem \ref{11.3}:
\begin{proof}
By Theorem \cite[II.3.7]{shh}(p.~297), there is a good
$\lambda^+$-frame, $\frak{s}$ such that its AEC is $(K,\preceq)$.
Now use Fact \ref{11.2}, where $\lambda^+$ stands for $\lambda$.
\end{proof}

%\begin{thebibliography}{99}

%\bibitem{Sh600} Saharon Shelah, {\sl Categoricity in abstract elementary classes: going up inductive step}, preprint.

%\bibitem{Sh:838} Saharon Shelah, {\sl Non-structure in $\lambda^{++}$ using instances of WGCH}, preprint.

%\end{thebibliography}

\end{document}